\documentclass[letter, 12pt]{article}

\usepackage[margin=1in]{geometry} 

\usepackage[utf8]{inputenc}
\usepackage{hyperref}
\hypersetup{
	unicode,
	linkcolor=blue, 
	pdfproducer={LaTeX},
	pdfcreator={pdflatex}
}

\usepackage{amsmath}
\usepackage{colonequals}
\usepackage{amsthm}
\usepackage{amssymb}
\usepackage{amsfonts}
\usepackage{amsxtra}   
\usepackage{amsrefs}
\usepackage{epsfig}
\usepackage{verbatim}
\usepackage{hyperref}
\hypersetup{
    colorlinks=true,
    linkcolor=blue,
    filecolor=magenta,      
    urlcolor=cyan,
}
\usepackage{enumitem}
\usepackage{tikz}
\usepackage{tikz-cd}
\usepackage{graphics}
\usepackage{graphicx, color}
\graphicspath{{fig/}}
\usepackage[all]{xy}
\usetikzlibrary{arrows,decorations.pathmorphing,decorations.pathreplacing,positioning,shapes.geometric,shapes.misc,decorations.markings,decorations.fractals,calc,patterns}
\usepackage[lined,boxed,commentsnumbered,ruled,linesnumbered]{algorithm2e}
\usepackage{float}

\usepackage{mathrsfs} 

\usepackage{lipsum}

\usepackage{booktabs} \usepackage[table,xcdraw]{xcolor}

\theoremstyle{plain}
\newtheorem{Theorem}{Theorem}
\newtheorem{Proposition}[Theorem]{Proposition}
\newtheorem{Corollary}[Theorem]{Corollary}
\newtheorem{Lemma}[Theorem]{Lemma}

\theoremstyle{definition}
\newtheorem{Definition}[Theorem]{Definition}

\newtheorem{Example}[Theorem]{Example}

\theoremstyle{remark}
\newtheorem{Remark}[Theorem]{Remark}

\usepackage[sort&compress,numbers,square]{natbib}
\newcommand{\newword}[1]{\textbf{\emph{#1}}}

\DeclareMathOperator{\rk}{rk}

\newcommand{\CC}{\mathbb{C}}

\newcommand{\KK}{\mathbb{K}}

\newcommand{\ZZ}{\mathbb{Z}}

\DeclareMathOperator{\GL}{GL}

\newcommand{\mm}{\mathfrak{m}}

\usepackage{bm}
\newcommand*{\allbf}[1]{\ifmmode\bm{#1}\else\textbf{#1}\fi}

\title{An algorithm for invariants of elementary abelian groups}

\author{Sasha Arasha$^1$, Marcus Cassell$^1$, Mal Dolorfino$^2$, Francesca Gandini$^1$, \\ Gordie Novak$^1$, Daniel Qin$^3$, Sumner Strom$^1$.}

\date{\footnotesize{
	$^1$St. Olaf College \\ Northfield, Minnesota, USA \\
    $^2$University of Washington \\ Seattle, Washington, USA \\
    $^3$University of California \\ Davis, California, USA \\}
}

\begin{document}
	\maketitle

\begin{abstract}
When we consider a finite abelian group acting linearly on a polynomial ring, we can find monomial generators for the subring of invariants. By Noether’s degree bound and Hilbert's finiteness theorem we know that there are finitely many minimal generators, but efficiently finding a generating set is not a trivial task. We present a new algorithm for computing the invariant ring for elementary abelian groups acting on polynomial rings with complex coefficients (or any other field of characteristic zero). We follow a two-step process: first we generate a collection of $n-k$ ``seed” invariants by calculating the kernel of a weight matrix that encodes our action. After we find the seeds, we ``grow” them into a generating set for the invariant ring by exploiting the lattice structure of invariants modulo $p$. Our algorithm performs better than the one currently available in \texttt{Macaulay2}, allowing us to compute more quickly invariants in this setting.

\medskip

\noindent\textbf{Keywords:} monomial invariants, diagonal linear action, elementary abelian groups, macaulay2.

\medskip

\noindent \textbf{MSC:} MSC2020 code 13A50.

\end{abstract}


\section{Introduction}~\label{sec:intro}

Let $G$ be an abelian group acting on a polynomial ring $\CC[x_1,\ldots,x_n]$ with complex coefficients. A polynomial $f\in\CC[x_1,\ldots,x_n]$ is \newword{invariant} under the action of $G$ if for all $g\in G$, $g\cdot f = f$. The set of invariant polynomials form a finitely generated $\CC-$algebra called the \emph{invariant ring}. Even though there are infinitely many possible generating sets for the invariant ring, by Noether's degree bound \cite{em}  we can find a finite minimal generating set of invariants of degree less than or equal to $|G|$, the order of $G$. Moreover, since we restrict our attention to finite abelian groups, after a change of basis every action is diagonal and we can find generating sets consisting of only monomials. We study how to find minimal monomial generating sets of the invariant ring arising from the action of a finite abelian group.

 In the case that $G$ is an elementary abelian $p$-group, we provide an effective algorithm to compute these generating sets. More precisely, for groups of the form $(\ZZ_p)^k$, we find an alternative to the algorithm by Derksen \cite{inv} and Gandini \cite{thesis} provided by a new two-step algorithm for generating a starting collection of invariant monomials (the \textit{seed invariants}) and then extending those (\textit{grow} them) into a full minimal generating set of invariants as discussed in Section~\ref{sec:results}. 

\begin{Theorem}\label{thm:main thm}
Let a finite abelian group $G$ isomorphic to $(\ZZ_p)^k$ act on some polynomial ring $R=\CC[x_1, x_2, \ldots, x_n]$ with $k<n$.  Then the entire minimal generating set for $R^G$ can be constructed using $n-k$ seed invariants.
\end{Theorem}


\section{Background}~\label{sec:background}

In this section, we use known results on invariant rings and minimal generating sets in order to study linear diagonal actions of finite abelian groups on polynomial rings. To make our approach precise, we now provide some needed background and tools for studying finite abelian group actions on polynomial rings.

\subsection{Representations of linear actions}
In this subsection, we provide the reader with some context for our setup. We will highlight how we connect group actions on a vector space and group actions on the ring of polynomial functions.
 
First, recall that a \textbf{(left) group action} of group $G$ on set $S$ is a map namely, 
$$G\times S \rightarrow S,\;\;(g,s)\mapsto gs$$
such that $es=s$ and $g(hs)=(gh)s$ for all $g,h\in G$ and $s\in S$.
If $S$ has the structure of a vector space and we have that $G$ respects the structure of $S$, then we have a \emph{linear action.} Formally this is as follows:

\begin{Definition}\label{Def: linear action}
Let $G$ be a group, $V$ a vector space over a field $\KK$. Then a \textbf{linear group action} is a group action $$G\times V\rightarrow V,\quad (g,v)\mapsto gv$$ where for all $v,w\in V$ and $c\in \KK$ we have that 
$$g(v+cw)=g v+cg w.$$
\end{Definition}
When doing computations of linear actions on finite-dimensional vector spaces, it is generally useful to consider the matrix representations of the group elements.\footnote{Recall that \emph{linear actions}, \emph{linear representations}, and \emph{matrix representations} are equivalent (up to some choice of basis).} In particular, for every linear action, one can find a matrix representation in the following manner. Choose a basis $\{e_1,\ldots,e_n\}$ of $V$ over $\KK$. Now a linear action of $G$ on $V$ is uniquely defined by the maps
$$(g,e_i)\mapsto ge_i$$ for $i\in[n]$ and $g\in G.$ In particular, for $g\in G$ and $i\in [n],$ $ge_i\in V$ so there exists coefficients $a_{ij}\in \KK$ such that $$ge_i=\sum_{j=1}^n a_{ij}e_j.$$ Then the matrix representation of $g\in G$ is $[a_{ij}]$.
\begin{Example}
Let $G = S_3$, so the group of permutations of three elements
$$S_3 = \{(),(1 \,2),(1\,3),(2\,3),(1\,2\,3),(1\,3\,2)\}. $$

An example of a matrix representation for the linear action for the group $G$ is given by
\begin{align*}
S_3&\mapsto
\{I, \begin{bmatrix}
    0 & 1& 0\\
    1 & 0& 0\\
    0 & 0& 1
\end{bmatrix},\begin{bmatrix}
    0 & 0& 1\\
    0 & 1& 0\\
    1 & 0& 0
\end{bmatrix},\begin{bmatrix}
    1 & 0& 0\\
    0 & 0& 1\\
    0 & 1& 0
\end{bmatrix},\begin{bmatrix}
    0 & 1& 0\\
    0 & 0& 1\\
    1 & 0& 0
\end{bmatrix},\begin{bmatrix}
    0 & 0& 1\\
    1 & 0& 0\\
    0 & 1& 0
\end{bmatrix} \}\\
&= \left\langle \begin{bmatrix}
0 & 1& 0 \\
0&0&1 \\
1&0&0
\end{bmatrix}, \begin{bmatrix}
0 & 1& 0 \\
1&0&0 \\
0&0&1
\end{bmatrix} \right\rangle.
\end{align*}
\end{Example}
Note that a matrix representation of a linear action ``represents'' the action as the standard matrix multiplication against a column vector. In particular, this matrix algebra perspective gives a natural extension of an action of $G$ onto its \emph{dual space} $V^*$, the set of linear functionals that map $V\rightarrow \KK$. We often represent these linear functionals as row vectors. This is a sensible choice since multiplying a row vector by a column vector encodes the evaluation pairing. 

\subsection{Abelian group actions on polynomial rings}
Recall that not only is the dual space $V^*$ a vector space, when $V$ is finite dimensional, $V\cong V^*.$ In particular, $V$ and $V^*$ are identified in the following manner: fixing an ordered basis $e_1,\ldots,e_n$ for $V$ immediately obtains a dual basis $x_1,\ldots,x_n$ for $V^*$ where $x_i(e_j)=\delta_{ij}$ for $1\leq i,j\leq n$ where $\delta_{ij}$ is the Kronecker delta function. This construction allows us to think of any action on $V$ also as an action on $V^*$. Let $G\times V\rightarrow V$ be a group action and let $e_1,\ldots, e_n$ be a basis for $V$. Then we can define a linear action on $V^*$ by 
$$G\times V^*\rightarrow V^*$$ 
by the point wise mapping
$$g\cdot x_i(e_j)=x_i(g^{-1}e_j), $$ 
which is then extended linearly. As an aside, for $G$ abelian, the action on the vector space and the dual space are $G$-isomorphic. 

The purpose of considering dual spaces is that they have additional structure coming from pointwise multiplication on the linear functionals. This gives us a \emph{ring of polynomial functions} in $n$ variables with over the field $\KK$, denoted $\KK[x_1,\ldots,x_n]$. If one does not specify a basis for the vector space $V$, we can also denote this ring as $\KK[V]$ and picking a basis gives us $\mathbb{K}[V]\cong \mathbb{K}[x_1,\ldots, x_n]$ where $x_1,\ldots,x_n$ is the dual basis to the chosen basis of $V$.

The standard addition and multiplication operations (on polynomials) agree with the pointwise addition and multiplication operations on the linear functionals.
Finally, we can define an action on $\mathbb{K}[V]$ by prescribing that the action preserves the multiplication structure. More precisely, consider the monomial $$\mathbf{x}^\mathbf{\alpha}=x_1^{\alpha_1} \cdots x_n^{\alpha_n},$$ 
where $\mathbf{\alpha}$ is an exponent vector. Then we can define a $G-$action on this monomial as follows: $$g\cdot \mathbf{x}^\mathbf{\alpha}=g(x_1)^{\alpha_1}\ldots g(x_n)^{\alpha_n}.$$

For the rest of the paper, assume that $G$ is a finite abelian group acting on a vector space $V$ of dimension $n$ over $\KK=\mathbb{C}$. Working over the infinite field $\mathbb{C}$ guarantees that $|G|$ is invertible (we call this the non-modular case) and moreover, guarantees the existence of a $|G|$-th root of unity. If in $G$ every element has order a power of $p$ (so $G$ is a $p$-group) and we let $|G|=s$, once we fix a primitive $s$-th root of unity $\mu_s$, then each element $g\in G$ can be mapped to an element in $\CC$ of the same order of $g$: specifically, $\mu_s^{\frac{|G|}{o(g)}}$.

\subsection{Abelian characters and weights}
Above, we discussed the importance of working over a field $\CC$ as this gives us a $|G|$-th root of unity. Mapping each group element to a root of unity is an example of a character of the group.
\begin{Definition}
A multiplicative character is a group homomorphism $\chi : G \rightarrow \mathbb{C}^{*}$ where $\mathbb{C}^{*} = \mathbb{C}\setminus\{0\}$.
\end{Definition}
The product of two multiplicative characters is a multiplicative character, so the set of multiplicative characters forms a group $\widehat{G}$. If the group homomorphism is an isomorphism, then $G \cong \widehat{G}$. We will consider actions with trivial kernel so that we can identify the two groups.

As $G$ is abelian, the action of $G$ on $V$ is diagonalizable and so it can be represented with a diagonal matrix. In fact, for $\rho: G \to V$ a linear representation of $G$, we have that all matrices $\rho(g)$ commute with each other and so they share eigenvalues and eigenvectors. For this reason, they are simultaneously diagonalizable, i.e., diagonalizable under the same change of basis. 


In particular, after we choose a basis for $V$ so that $\mathbb{C}[V] = \mathbb{C}[x_1,\ldots,x_n]$, then  for each $g \in G$, 
$$g\cdot x_j = \chi_j(g)x_j,$$ 
where $\chi_j: G \rightarrow \mathbb{C}^{*}$ is a multiplicative character (i.e., it has the property that $\chi_j(gh) = \chi_j(g)\chi_j(h)$).

Moreover, for any finite abelian group, we can choose an isomorphism 
$$G \cong \mathbb{Z}/(d_1) \times \ldots \times \mathbb{Z}/(d_r)$$ 
with $d_i \mid d_{i+1}$ for all $i$. This choice allows for the minimal number of cyclic factors and so the minimal number of generators for $G$.
 In particular, if we identify $G$ with the group $\mathbb{Z}/(d_1) \times \ldots \times \mathbb{Z}/(d_r)$, we can use additive notation for the group structure in $G$. After this identification, a set of generators of $G$ is
\begin{align*}
    g_1 &= (1,0,...,0) \\
    g_2 &= (0,1,...,0) \\
    \vdots \\
    g_r &= (0,0,...,1) \\
\end{align*}
where $g_i$ has order $d_i$ for $1 \leq i \leq r$.

Now let $\mu_i \in \mathbb{C}^{\star}$ be a $d_i-$th root of unity. The diagonal action of $g_i$ on $x_j$ is given by 
$$g_i \cdot x_j = \chi_j(g_i)x_j = \mu_i^{w_{ij}}x_j$$ for some $w_{ij} \in \mathbb{Z}/(d_i)$. We call $w_{ij}$ the \newword{weight} of the action $g_i$ on $x_j$. The integer vector $$\mathbf{w}_j = (w_{1j},\ldots,w_{rj}) \in \mathbb{Z}/(d_1) \times \ldots \times \mathbb{Z}/(d_r)$$ is the weight associated to $\chi_j$, or the weight of the action $G$ on the variable $x_j$.Similarly, we will define
$$\mathbf{w}^i = (w_{i1},\ldots,w_{in}) \in (\mathbb{Z}/(d_i))^n,$$ 
to be the weight of the action $g_i$ on $V$.  We collect all weights in the \newword{weight matrix} $W= [w_{i,j}]_{i,j}$, whose rows are the weights of the generators of $G$ and whose columns are the weights of the generators of $R$.

If the action has trivial kernel, by mapping multiplicative characters to their weight, we obtain an isomorphism between the group of multiplicative characters $\widehat{G}$ and the group $\mathbb{Z}/(d_1) \times \ldots \times \mathbb{Z}/(d_r) \cong G$, showing explicitly that $\widehat{G}$ and $G$ are isomorphic.

\subsection{Invariant rings}
Let $M$ denote the set of non-constant monomials in $R=\mathbb{C}[x_1,\ldots,x_n]$. Then, since the action of the group $G$ is diagonal, each monomial $m \in M$ is mapped to a scalar multiple of itself.
\begin{Definition}
A polynomial $f \in R$ is called \newword{invariant} under the action of $G$ if for all $g\in G$, we have that $g\cdot f = f$.
\end{Definition}
Notice that a polynomial will be invariant if and only if each of its monomials is. As a result, we can choose a set of invariant monomials as a minimal generating set for the invariant ring.
\begin{Definition}
The set of all invariant polynomials in $n$ variables with complex coefficients under the linear action of a finite group $G$ forms a ring called the \newword{invariant ring}, which we will denote by $R^G$.
\end{Definition}

For the action of a finite group on a polynomial ring with complex coefficients, Noether \cite{em} proved a bound on the degree of the generators of the invariant ring.
\begin{Theorem}[Noether's Degree Bound, Noether 1915] 
For a finite group $G$ over $\mathbb{C} $, every invariant ring $R^G$ can be generated by a finite set of invariants. In particular, every invariant ring can be generated by a \newword{minimal generating set} $\mathcal{M}$ with $|\mathcal{M}| \leq |G|$. 
\end{Theorem}





As previously mentioned, our primary objective is to find ``faster'' or different constructions of minimal generating sets of invariants, perhaps by adding minimal conditions (or assumptions). However, this can be fairly challenging. Consider the following example. 
\begin{Example}
    Let $$G = \left<\begin{bmatrix}-1 & 0 \\
0 & 1 \\\end{bmatrix} , \begin{bmatrix}1 & 0 \\
0 & -1 \\\end{bmatrix}\right> \cong \ZZ_2 \times \ZZ_2$$ act on the polynomial ring $\mathbb{C}[x_1,x_2]$. We can immediately verify that $x_1$ and $x_2$ are not invariant under $G$. Let $g_1 = \begin{bmatrix}-1 & 0 \\
0 & 1 \\\end{bmatrix}$ and $g_2 = \begin{bmatrix}1 & 0 \\
0 & -1 \\\end{bmatrix}$. Then $g_1 \cdot x_1 = -x_1$ and $g_2 \cdot x_2 = -x_2,$ so $x_1$ and $x_2$ are not invariant. But since $g_i \cdot x_1^nx_2^m = (g_i\cdot x_1)^n(g_i\cdot x_2)^m,$ we see that $x_1^2$ and $x_2^2$ are invariant. Moreover, any monomial of the form $x_1^nx_2^m$ where $n,m$ are even is invariant, whereas if either $n$ or $m$ is odd, the monomial $x_1^nx_2^m$ is not invariant. Thus, the invariant ring under this group action consists of all polynomials where each term is of the form $\lambda x_1^nx_2^m$ for $\lambda \in \mathbb{C}$ and $n,m$ non-negative even integers. Therefore, a minimal generating set for this invariant ring is $\mathcal{M} = \{x_1^2,x_2^2\}$.
\end{Example}
From this example, we see that it becomes complicated and time-consuming to determine minimal generating sets as the order of the acting group and the size of the polynomial ring grow. Consequently, the purpose of our research is to find ways to algebraically and geometrically construct minimal generating sets of invariant rings to improve upon existing algorithms and implementations.

Luckily, Noether's degree bound is achieved only when the group is cyclic (as proved by Schmid \cite{schmid}). In particular, in the case of elementary abelian groups, Olson \cite{ol1} showed that there is a better bound on the degree of the generators.
\begin{Theorem}[Olson's Degree Bound, Olson 1969] 
Let $G$ be an elementary abelian group, so that $G \cong (\ZZ/p)^k$. For a linear action of $G$, we can find minimal generators for the ring of invariants of degree at most $k(p-1) + 1$.
\end{Theorem}
For this reason, considering elementary abelian groups has the potential to lead to faster methods of computing invariants. We introduce a two-step algorithm that achieves this goal.

\section{Results}~\label{sec:results}
Our algorithm for computing minimal generating sets of invariant rings of elementary abelian groups proceeds in 2 stages. First, with \emph{seed generation}, we find a collection of $n-k$ invariant monomials called \emph{seeds} whose exponents generate the kernel of the weight matrix $W$. The second stage we call the \emph{seed growth}, where we expand our small collection into a minimal generating set for the ring of invariants.

We first illustrate a key idea of our method by considering the action of $G=\ZZ_p\times\ZZ_p$ (for $p$ a prime) on the polynomial ring $R=\CC[x,y,z]$. Our target object is a minimal generating set of the invariant ring $R^G$. With one invariant monomial $m$ and the monomials  $x^p, y^p$, and $z^p$, we will be able to construct the rest of a minimal generating set. We call $m$ a \newword{seed invariant}. 

\begin{Example}~\label{ex: 2by3}
Consider the group $G\cong \ZZ_3\times\ZZ_3$ acting on $\CC[x,y,z]$ defined by the weight matrix
$$W=\begin{bmatrix}
1 & 0 & 1\\
0 & 1 & 1 
\end{bmatrix}.$$

Using the Derksen-Gandini algorithm \cite{thesis}, we find a minimal generating set of monomials: 
$$\{x^3,y^3,z^3, xyz^2,x^2y^2z\}.$$ 

Observe that the terms with mixed support have nontrivial algebraic relations as follows: $$(xyz^2)^2=x^2y^2z^4=z^3\cdot(x^2y^2z),$$ $$(x^2y^2z)^2=x^3\cdot y^3\cdot (xyz^2).$$

With the seed invariant $m=xyz^2$ and the pure powers $x^3,y^3,z^3$ we obtain the whole minimal generating set by considering all the powers on $m$ modulo 3, as $\overline{m^2}=\overline{x^2y^2z^4}=x^2y^2z$.
\end{Example}

This behavior is not unique to Example~\ref{ex: 2by3}. In fact, it is the primary observation that motivates the algorithm. What follows is then a two-fold process: (1) describe a method to find a small collection of seeds that can grow into a full generating set (seed generation)  and (2) expand the seeds by taking algebraic combinations of them until we get a set of minimal generators (seed growth). We start with the second step as it was discovered first, and then discuss how to accomplish seed generation.

\subsection{Seed growth}~\label{subsec: seed growth}
Say $G$ is $p$-abelian, i.e., $G\cong (\ZZ_p)^k$ and acts on the polynomial ring $\mathbb{C}[x_1,\ldots,x_n]$ and let $\mu_p\in\mathbb{C}$ denote some $p$th root of unity. Notice that there are invariants which we can immediately identify by  knowing the order of $G$, namely $x_j^p$ for all $j$. We call these monomials the \newword{pure powers}.

\begin{Lemma}\label{lem: purepowers 2by3}
Let $G\cong (\ZZ_p)^k$ be a group acting on $\CC[x_1,\ldots,x_n]$. Then $x_j^p$ is invariant for $j\in[n]=\{1, \ldots , n\}$. 
\end{Lemma}
\begin{proof}
Since $G$ is $p$-abelian, there exists some weight matrix $W=(w_{ij})$ defining the action of $G$ as $g_i x_j=\mu_p^{w_{ij}}x_j$ where $\mu_p$ is a $p$th root of unity, i.e., $\mu_p^p=1$. Since $g$ is a homomorphism, expanding along gives us the desired result:
$$g_i(x_j^p)=(g_i(x_j))^p={\mu_p^{p \, w_{ij}} x_j^p}=(\mu_p^p)^{w_{ij}}x_j^p=x_j^p.$$ Therefore, any pure power monomial is invariant.
\end{proof}
We then obtain the following corollary on products of pure powers:
\begin{Corollary}~\label{cor: purepowers 2by3}
For $I \subseteq [n],$ $x^{p_I}=\prod_{j\in I}x_j^p$ is invariant under the action of $G$. 
\end{Corollary}
More generally, if we allow for rational functions, we have the following:
\begin{Corollary}
    For any $\beta\in \ZZ^n$, $x^{p\beta}$ is invariant as an element of $\operatorname{Frac}(R)^G$.
\end{Corollary}
Albeit simple, we can use these corollaries in a reduction mod $p$ fashion to cook up other invariants. In particular, observe the following: 
\begin{Proposition} 
    Let $G$ be elementary $p-$abelian, $R=\CC[x_1,\ldots,x_n],$ $0\leq \alpha_1,\ldots,\alpha_n<p$, and for $I\subseteq [n]$, let $p_I=\sum_{j\in I}pe_j.$ If  $x^\alpha=x_1^{\alpha_1}\cdots x_n^{\alpha_n}$ is in $R^G$, then for all $I$ containing $\operatorname{supp}(\alpha),$ $x^{p_I-\alpha}$ is in $R^G.$ 
\end{Proposition}

\begin{proof}
    Observe that $$g_i(x^{p_I-\alpha})x^\alpha=g_i(x^{p_I-\alpha})g_i(x^\alpha)=g_i(x^{p_I})=x^{p_I}=x^{p_I-\alpha}x^{\alpha}.$$
    Therefore, $$g_i(x^{p_I-\alpha})=x^{p_I-\alpha}.$$
\end{proof}

More generally, given an invariant $x^{\alpha}$, we obtain a lattice of invariants from which we can extract nonnegative coordinates.
\begin{Proposition}
If $x^\alpha\in R^G$ then for all $\beta\in\ZZ^n$, $x^{\alpha+p\beta}\in Frac(R)^G.$ In particular, if $\alpha+p\beta\in\ZZ_{\geq0}^n,$ $x^{\alpha+p\beta}\in R^G$.
\end{Proposition}

More generally, for any fixed $\alpha$ associated to an invariant monomial, we obtain a family of lattices.
\begin{Proposition}
    If $x^\alpha\in R^G$ then, for all $\beta\in \ZZ^n$ and $k\in \ZZ$ such that $k\alpha+p\beta \in\ZZ_{\geq 0}^n$, $x^{k\alpha+p\beta}$ is in the invariant ring.
\end{Proposition}
Moreover, for $\alpha$ large, we can step along the lattice to find candidates for a minimal generating set. 
\begin{Proposition}~\label{prop:p-piping}
    Let $x^\gamma$ be a monomial in the invariant ring $R^G$. Then there uniquely exists coordinates $t\in [0,p-1]^n$ and $\beta\in \ZZ^n$ such that $\gamma=p\beta+t$. In particular, $x^t\in R^G$. 
\end{Proposition}
\begin{proof}
    Reduction modulo $p$ gives the desired unique decomposition of $\gamma$
    $$\gamma =p\beta +t\quad\text{with}\quad  t\in [0,p-1]^n.$$
    By the same arguments as above, $x^t$ is invariant.
\end{proof}
By Proposition~\ref{prop:p-piping}, every invariant monomial can be expressed as above. In particular, every element of a minimal generating set appears in this fashion. As a result, if we have a set of seed invariants, we can grow them in a set of minimal generators for the invariant ring. Next, we will discuss how to find such initial seeds with our seed generation algorithm. 

\subsection{Seed generation}
In Subsection~\ref{subsec: seed growth}, we described a method to recover a minimal generating set given a set of seeds. What remains is an algorithm to find a set of seeds. This is equivalent to finding a $\ZZ-$basis for the associated lattice of $R^G.$ 

\begin{Remark}
    For any $G=\ZZ_p^k$ acting on $R=\CC[x_1,\ldots,x_n]$, there exists a set of seeds $\mathcal{S}$ consisting of monomials. Equivalently, every associated lattice to some $R^G$ has a $\ZZ$-basis consisting of nonnegative coordinates.
\end{Remark}

Informally, if we run into any negative coordinates, we can add $p$ until we reach non-negative values. Formally, given an element $x^\alpha$ of a fixed $\ZZ-$basis $\mathcal{B}$, if $\alpha$ has negative coordinates on $I\subseteq[n],$ observe that $\widetilde{\alpha}:=\alpha+ \sum_{i \in I} (\alpha_i + c_i \,  p) \, e_i$ has nonnegative coordinates for large enough $c_i$. A quick check confirms that $\mathcal{B}-x^\alpha+x^{\widetilde{\alpha}}$ is still a $\ZZ-$basis. 
\begin{Definition}[Plucker coordinates]
    We define the \textit{plucker} (coordinate) $p_S$ on a matrix $M$ as the determinant of square submatrix of $M$ obtained by choosing the columns in the index set $S$. For any index $i$ in the index set $T$, we also define $p_{\underline i}$ to be the plucker $p_S$ associated to the index set $S=T \setminus\{i\}$.
\end{Definition}

\begin{Proposition}\label{construct-seeds}
Assume that the columns of $W$ are linearly independent over $\mathbb Z$. Find an index set $T\subset\{1,\dots,m\}$ with $|T|=n$ such that the square sub‑matrix $W_T$ has non-zero determinant. For every $i\notin T$ put $S_i=T\cup{i}$. Define $v_i\in\mathbb Z^{m}$ by the formula
$$ v_i = \sum_{a \in S_i}(-1)^{j(a)+1}e_{a}p_{\underline{a}}, \quad \textnormal{where $j(a)$ is the position of $a$ inside $S$.}$$
Then, \begin{enumerate}
    \item $Wv_i= \vec 0$ for all $i\notin T$, so $v_i\in\ker W$.
    \item The $m-n$ vectors $\{ v_i\mid i\notin T \}$ are linearly independent. 
\end{enumerate}
Consequently $\{ v_i\mid i\notin T \}$ is a $\mathbb Z$‑basis of $\ker W$, and $\dim(\ker W)=m-n$.
\end{Proposition}

\begin{proof}
    First, we show that for all $i \not\in T,$ $Wv_i = 0$ by showing that for all row vectors $r_k \in W,$ $r_k \cdot v_i = 0.$ Letting $r_{i,k}$ denote the restriction of the vector $r_k$ to the columns indexed by $T \cup i,$ we have that \[v_i \cdot r_k = \det\begin{pmatrix}
        r_{1,i} \\ 
        r_{2,i} \\
        \vdots \\
        r_{n,i} \\
    \end{pmatrix} = 0,\] hence $Wv_i = 0.$
    
    Now, we prove (2), and we note that the conclusion that the $\{v_i\}$ form a basis for $\ker W$ follows from (1), (2), and the rank-nullity theorem. Let $v_k$ and $v_{\ell}$ be two vectors in this set with $k \neq \ell$. Observe that $k \not\in T \cup {\ell}$ and $\ell \not\in T \cup k.$ Moreover, since $p_T=\det(W_T) \neq 0,$ $p_{\underline{i}} \neq 0$ for all $i \not\in T.$ Therefore, $v_{k}$ must have a nonzero value in position $k$, while $v_{\ell}$ must have a zero in position $k.$ Thus, since $v_k$ has a nonzero $k$-th coordinate while all other $v_i$ have zero $k$-th coordinate, the set $\{v_i\}$ must be linearly independent over $\ZZ$.
\end{proof}

\begin{Example}\label{ex:seed-construction}
Let $W = \begin{pmatrix}
    1  & 1 &1 &1 \\
    0 & 0 & 1 & 1
\end{pmatrix}$. We want to find a $\ZZ$-basis of $\ker W$. We see that submatrix $W_{23}$ has a non-zero determinant.
We will extend the submatrix $W_{23}$ into a matrix with nullity 1, by adding the other column vectors, which yields two submatrices, namely $W_{123},$ and $W_{234}$. We compute the corresponding seeds in this example.
\begin{align*}
    v_{W_{123}} &= \sum_{a_i\in\{1,2,3\}}(-1)^{i+1}e_{a_i}p_{\underline{a_i}} \\
    &= e_1\begin{vmatrix}
        1 & 1 \\ 0 & 1
    \end{vmatrix} -e_2 \begin{vmatrix}
        1 & 1 \\ 0 & 1
    \end{vmatrix} +e_3\begin{vmatrix}
        1 & 1 \\ 0 & 0
    \end{vmatrix} = \begin{pmatrix}
        1 \\ -1 \\ 0 \\ 0
    \end{pmatrix}.
\end{align*}
Whilst,
\begin{align*}
    v_{W_{234}} &= \sum_{a_i\in\{1,2,3\}}(-1)^{i+1}e_{a_i}p_{\underline{a_i}} \\
    &= e_2\begin{vmatrix}
        1 & 1 \\ 1 & 1
    \end{vmatrix} -e_3 \begin{vmatrix}
        1 & 1 \\ 0 & 1
    \end{vmatrix}+e_4 \begin{vmatrix}
        1 & 1 \\ 0 & 1
    \end{vmatrix} = \begin{pmatrix}
        0 \\ 0 \\ -1 \\ 1
    \end{pmatrix}.
\end{align*}
Lastly, we reduce mod p to create our basis for our kernel:
$$v_{W_{123}}=(1,p-1,0,0), \quad v_{W_{234}}=(0,0,p-1,1).$$ 

Observe that our method places zeros in places such that our basis is always linearly independent. In particular, we have that $$W \, v_{W_{123}} = (p,0) \equiv (0,0) \pmod p, \quad W \, v_{W_{234}} = (p,p) \equiv (0,0) \pmod p.$$ So our seeds are invariant.
\end{Example}

\begin{Remark}
    In the above situation and with the notation of Proposition \ref{construct-seeds}, since we assumed the $n$ rows of $W$ are linearly independent over $\ZZ$, a subset $T$ always exists. 
\end{Remark}

\subsection{Connecting invariants to the kernel}

The elements of the kernel of the weight matrix give us invariant monomials. If $W$ is a full rank $n\times k$ matrix, we get $n-k$ seeds which are linearly independent of each other, and so give us minimal monomial generators for the invariant ring. Notice that we also get $n$ other monomials, the pure powers, which are also minimal, as their support consists of only one variable. We want to take these $n+(n-k)$-elements and show there exists a $\ZZ$-basis of the sub-lattice within them.

    \begin{Lemma}~\label{lem: inv=ker}
        Let $W$ be full rank $n\times k$ weight matrix representing the action of $G=(\ZZ/p)^k$ acting on $R=\CC[x_1,\ldots,x_n].$ Then observe that the monomial $x^\alpha$ is invariant if and only if $\overline{\alpha}\in \ker W_{\ZZ/p}$.
    \end{Lemma}
    \begin{proof}
        Recall that $x^\alpha$ is invariant if and only if for all $i\in[k],$ $g_i x^\alpha=x^\alpha.$ Let $W_i$ denote the $i$th row of $W.$ Then notice that $$g_i x^\alpha =\mu_p^{W_i\cdot \alpha}x^\alpha=x^\alpha \quad \text{if and only if}\quad W_i\cdot \alpha \equiv 0\pmod{p}\quad\text{for all $i$}.$$
        The right-hand side becomes the desired result when put into matrix language. 
    \end{proof}

        \begin{Proposition}\label{intgenset}
        If $W$ is an $n\times k$ matrix, we get a set of $\dim ker(W_{\ZZ/p})$ many seeds $\mathcal{S}$ and $\mathcal{S}\cup\{pe_i:i\in[n]\}\subseteq \ZZ^n$ is a $\ZZ-$generating set of $N.$
    \end{Proposition}
    \begin{proof}
        By Lemma~\ref{lem: inv=ker}, $\alpha \in N$ if and only if its reduction modulo $p$ is in the kernel of $W_{\ZZ/p}.$ In particular, if $\alpha \in N$ then  $\alpha=\overline{\alpha}+p\beta$ for some $\beta\in\ZZ^n$ such that $\overline{\alpha}\in\ker W_{\ZZ/p}$, i.e., $\overline{\alpha}=\sum_{i=1}^{n-k} c_i\alpha_i$ for $0\leq c_i\leq p-1.$ So $\mathcal{S}\cup \{pe_i\}$ generates $N.$
    \end{proof}
    
        Proposition~\ref{intgenset} gives us that the cardinality of a set of seeds is equal to the dimension of $\ker(W_{\ZZ/p})$. We can translate our linearly independent seeds in monomial exponents by considering them just as an integer exponents. We show that going from $(\ZZ/p)^n$ to $\ZZ^n$, our seeds stay linearly independent. More formally, we prove the following proposition.
      \begin{Proposition}~\label{prop: Gauche}
        Let $\mathcal{S}\subseteq (\ZZ/p)^n$ be a $\ZZ/p$-linearly independent set. Then $\mathcal{S}$ as a subset of $\ZZ^n$ under the embedding 
        \footnote{Such an embedding ``forgets'' some of the obvious structure and it is referred to as a ``Gauché'' mapping. 
        }
        $(\ZZ/p)^n\cong [0,p-1]^n\hookrightarrow \ZZ^n$ is a $\ZZ-$ linearly independent set with respect to $\ZZ^n$.
    \end{Proposition}
    \begin{proof}
        Say $\mathcal{S}\subseteq (\ZZ/p)^n$ is $\ZZ/p-$linearly independent and assume for contradiction that under the Gauché embedding it is not $\ZZ$-linearly independent, i.e., there exists $c_\alpha\in\ZZ$ such that $$\sum_{\alpha\in\mathcal{S}}c_\alpha \alpha=0.$$
        Now, if there exists some $\alpha$ where $c_\alpha\not\equiv 0\pmod{p}$, then immediately reducing modulo $p$ gives us a nonzero  $\ZZ/p$-linear combination of zero, i.e., $\mathcal{S}$ is not $\ZZ/p$-linearly independent, which is a contradiction. On the other hand, if all $c_\alpha\equiv 0\pmod{p}$, then for $c_\alpha \neq 0$ we can write $c_\alpha= p^{k_\alpha}\cdot m_\alpha$ where $k_\alpha\geq 1$ and $p\nmid m_\alpha.$ Let $k=\min\{k_\alpha:\alpha\}$. Then, observe that 
        $$\sum_{\alpha\in\mathcal{S}}\frac{c_\alpha}{p^k} \alpha=0$$
        with $c_\alpha/p^k\in \ZZ$ and there is some $\alpha $ with $c_\alpha \not\equiv 0\pmod{p}$.  Again, reducing modulo $p$ gives a nontrivial $\ZZ/p$-linear combination of zero which is a contradiction.
    \end{proof}

\section{Future Directions}~\label{sec:future}
The results in Section ~\ref{sec:results} allow us to create a new algorithm for elementary abelian groups, so groups of the type $(\ZZ/p)^k$. Naturally, one might want to consider other $p-$groups or at least groups of the type $(\ZZ/p^m)^k$ for $m>1$. In this section we will start in that direction by showing  how we can use Nakayama's Lemma to connect our elementary case to the $(\ZZ/p^m)^k$ case, as well as discuss some related reductions.

\subsection{Faithful representations}
First, we will address a simple reduction. If $\phi: \ZZ/(p^k)\rightarrow \GL(V)$ is not faithful, i.e., $H=\ker\phi\neq \{e\}$, then the weight matrix $W$ has linearly dependent rows. In this case, one can remove rows until we obtain a faithful representation of $G/H$ on $V$ and a new full rank weight matrix $W'$ with  $\rk(W)=\rk(W')$. Since the invariant ring depends only on the kernel of the weight matrix, this reduction does not change $R^G$. Thus, we may assume without loss of generality that the representation is faithful.

\subsection{Extending to other \textit{p}-groups}

For groups $G\cong (\ZZ/p^m)^k$ with $m>1$. Observe that the ring $R=\ZZ_{p^m}$ is a local ring, and so Nakayama's Lemma applies.

\begin{Lemma}[Nakayama's Lemma]
    Let $M$ be a finitely generated module over a local ring $R$ with maximal ideal $\mm$. If $M/\mm M$ is a vector space over $R/\mm$, then a basis of $M/\mm M$ lifts to a minimal set of generators of $M$.
\end{Lemma}

We now apply this in the context of weight matrices over $\ZZ_{p^k}$.
\begin{Proposition}
    Let $R = \ZZ_{p^k}$ with maximal ideal $\mm=(p)$, and let $A$ be a weight matrix with entries in $R$, say $A: R^n \to R^m$. Let $\bar{A} = A \pmod{p}$ denote the reduction of $A$ modulo $p$. Then
    $$\mathrm{rank}(\ker(A)) \leq \mathrm{rank}(\ker(\bar{A})),$$
    and the generators of $\ker(A)/(p) \subseteq \ker(\bar{A})$ can all be lifted to generate $\ker(A)$.
\end{Proposition}

\begin{proof}
    Let $R = \ZZ_{p^m}$ with $\mm=(p).$ Since $\ZZ_{p^m}$ is a local ring with unique maximal ideal $\mm$, Nakayama's Lemma applies. Let $M = \ker(A)$. We have $R/\mm = \ZZ_{p^m}/(p)\cong \ZZ_p$, so $M/\mm M$ is a vector space over $R/\mm$.

    We have the exact sequence
    $$ 0 \to M \to R^n \overset{A}\to R^m.$$
    Reducing modulo $\mm$ gives
    $$0 \to M/\mm \to R^n/\mm \overset{\bar A}{\to} R^m/\mm,$$
    which is isomorphic to
    $$0 \to M/\mm \to \ZZ_p^n \overset{\bar A}{\to} \ZZ_p^m.$$
    By Nakayama's Lemma, the generators of $M/\mm$ lift to generators of $\ker(A)$, so $\mathrm{rank}(\ker(A))=\mathrm{rank}(M/\mm)$.

    Now observe that the map $\pi:\ker(A) \to \ker(\bar{A})$, defined by $\pi(a) := a \pmod{p}$, has kernel $\mm M = (p)M$. By the First Isomorphism Theorem, the image of $\pi$ is isomorphic to $M/\mm M$, so $M/\mm \subseteq \ker(\bar{A})$ and $\mathrm{rank}(M/\mm) \leq \mathrm{rank}(\ker(\bar{A}))$. Therefore
    $$\mathrm{rank}(\ker(A)) = \mathrm{rank}(M/\mm) \leq \mathrm{rank}(\ker(\bar{A})).$$
    Hence the generators of $\ker(A)/(p) \subseteq \ker(\bar{A})$ lift to generate $\ker(A)$.
\end{proof}

This result suggests a strategy for computing invariants of $(\ZZ_{p^m})^k$: reduce the entries of the weight matrix mod-$p$, use our algorithm, and then lift the invariants back into $\ZZ_{p^m}$. We currently do not have an efficient lifting algorithm which would be necessary to be able to implement this strategy.

\section{Implementation}~\label{sec:implementation}

The algorithms described above have been implemented in the computer algebra system \href{https://macaulay2.com/}{\texttt{Macaulay2}} \cite{m2}. We use \texttt{Macaulay2} as this is an open source, community-based, and powerful computer algebra software designed for research in algebraic geometry and commutative algebra, capable of efficiently computing with ideals, modules, and Gröbner bases. Its built-in algebraic objects make it especially convenient for exploring polynomial rings.

The algorithms are implemented with the new \texttt{Elementary} strategy in the package \texttt{InvariantRing} \cite{invm2}, which is used for an elementary abelian group whose weight matrix is full rank.  We include the pseudocode for the algorithms in the subsection Pseudocode \ref{code}. We tested the new strategy against the one available for finite abelian groups (\texttt{DerksenGandini}) for weight matrices of sizes $2 \times 3$, $2 \times 4$, $2 \times 5$ and primes up to 29 (as the older strategy already takes nearly an hour at that point). The new strategy performs faster than the older one, see Table \ref{fast}, Table \ref{faster}, Table \ref{fastest}. All computations were performed on a Dell XPS laptop with an Intel Core i9-12900HK processor running Macaulay2 1.26.05 on Fedora Linux 43 and version 2.4 of the InvariantRing package

\begin{table}[H]
\centering
\begin{tabular}{@{}|l|l|l|@{}}
\toprule
\rowcolor[HTML]{FFFFFF} 
{\color[HTML]{333333} $p$} & {\color[HTML]{333333} DerksenGandini (seconds)} & {\color[HTML]{333333} Elementary (seconds)} \\ \midrule
2                          & 0.0188904                                       & 0.00240308                                  \\ \midrule
3                          & 0.0339419                                       & 0.00225972                                  \\ \midrule
5                          & 0.0871715                                       & 0.00187862                                  \\ \midrule
7                          & 0.240657                                        & 0.0017983                                   \\ \midrule
11                         & 0.930497                                        & 0.00207469                                  \\ \midrule
13                         & 1.4655                                          & 0.0020195                                   \\ \midrule
17                         & 3.67577                                         & 0.0023383                                   \\ \midrule
19                         & 5.64084                                         & 0.00334087                                  \\ \midrule
23                         & 11.021                                          & 0.0032336                                   \\ \midrule
29                         & 25.6113                                         & 0.00369017                                  \\ \bottomrule
\end{tabular}
\caption{Run times for $2\times 3$ weight matrix for primes up to 29.} \label{fast}
\end{table}

\begin{table}[H]
\centering
\begin{tabular}{@{}|l|l|l|@{}}
\toprule
\rowcolor[HTML]{FFFFFF} 
{\color[HTML]{333333} $p$} & {\color[HTML]{333333} DerksenGandini (seconds)} & {\color[HTML]{333333} Elementary (seconds)} \\ \midrule
2  & 0.0194246               & 0.00251401          \\ \midrule
3  & 0.0403208               & 0.00318608          \\ \midrule
5  & 0.165486                & 0.00267351          \\ \midrule
7  & 0.664042                & 0.00343511          \\ \midrule
11 & 5.85124                 & 0.00497864          \\ \midrule
13 & 14.5446                 & 0.00602303          \\ \midrule
17 & 76.129                  & 0.00924483          \\ \midrule
19 & 160.291                 & 0.0122608           \\ \midrule
23 & 554.185                 & 0.0147222           \\ \midrule
29 & 2468.83                 & 0.0203569           \\ \bottomrule
\end{tabular}
\caption{Run times for $2\times 4$ weight matrix for primes up to 29.} \label{faster}
\end{table}

\begin{table}[H]
\centering
\begin{tabular}{@{}|l|l|l|@{}}
\toprule
\rowcolor[HTML]{FFFFFF} 
{\color[HTML]{333333} $p$} & {\color[HTML]{333333} DerksenGandini (seconds)} & {\color[HTML]{333333} Elementary (seconds)} \\ \midrule 
2  & 0.0195961               & 0.00266178          \\ \midrule
3  & 0.0431                  & 0.0032897           \\ \midrule
5  & 0.264658                & 0.00676409          \\ \midrule
7  & 0.817787                & 0.0141474           \\ \midrule
11 & 7.40832                 & 0.0454255           \\ \midrule
13 & 20.6554                 & 0.0770973           \\ \midrule
17 & 97.6529                 & 0.265525            \\ \midrule
19 & 211.821                 & 0.239991            \\ \midrule
23 & 733.398                 & 0.531622            \\ \midrule
29 & 3410.04                 & 1.12467             \\ \bottomrule
\end{tabular} 
\caption{Run times for $2\times 5$ weight matrix for primes up to 29.} 
\label{fastest}
\end{table}

\subsection{Pseudocode}~\label{code}

\newcommand{\n}{\vspace{0.1cm}\\}
\newcommand{\en}{\vspace{0.2cm}\\}
\newcommand{\vn}{\vspace{-0.2cm}\\}
\newcommand{\tn}[1]{\textnormal{#1}}
\newcommand{\tb}[1]{\textnormal{\textbf{#1}}}
\newcommand{\ceq}{\;\tn{\texttt{==}}\;}
\newcommand{\cneq}{\;\tn{\texttt{!=}}\;}
\newcommand{\code}[1]{\tn{\texttt{#1}}}

\subsubsection{Seed generation}
The pseudocode in Algorithm \ref{alg1} is a simplification of the seed generation algorithm implemented in \texttt{Macaulay2} (M2).

\begin{algorithm}[h!]
    \SetAlgoLined
    \KwData{\\
        $\quad R$, a polynomial ring\\
        $\quad W$, a weight matrix with $n$ rows\\
        $\quad p$, the order of the group, $\mathbb Z$\\
    }
    \KwResult{\\
        $\quad L$, a list of the polynomial seed invariants\vspace{0.3cm}\\
    }

    \DontPrintSemicolon
    $w=\begin{bmatrix} \,\,\end{bmatrix}$ \tcc*{Set $w$ to an empty matrix}
    \ForEach {$n\times n$ \tn{submatrix} $s\in W$}{
        \If {$\code{Det}(s)\cneq 0$}{
            $w = s$\;
            \tb{break}\tcc*{Breaking ends the for loop}
        }
    }
    \BlankLine
    $\tb{a} = \{\,\}$\tcc*{creates list to store vectors}
    \ForEach{$\vec{v}\in W$ \tn{s.t.} $\vec{v}\notin w$}{
        $Q = \{\,\}$\tcc*{empty list to store vector results}
        $j = 1$\;
        \BlankLine
        \ForEach{$\vec x \in [\vec v \,\,\,\, w]$}{
            $Q$.\texttt{append}($j \cdot e_{\vec v} \cdot \texttt{Det}(\,[\vec v \,\,\,\, w ]\,\backslash \,\vec x\,$))\;
            $j=j\cdot -1$
        }
        \BlankLine
        $\tb{a}$.\texttt{append}($\sum Q$)
    }
    \BlankLine
    \tb{return }\texttt{toPolynomial(\tb{a}, $p$, $R$)}  
    \caption{Seed generation} \label{alg1}
\end{algorithm}

\subsubsection{Seed growth}
The pseudocode in Algorithm \ref{alg2} is a simplification of the seed growth algorithm implemented in \texttt{Macaulay2} (M2).

\begin{algorithm}[h!]
 \caption{Seed growth} \label{alg2}
    \SetAlgoLined
    \KwData{\\
        $\quad L$, a list of polynomial seed invariants\\
        $\quad R$, the ring of the polynomial invariants\\
        $\quad p$, the order of the group, $\mathbb Z$\\
    }
    \KwResult{\\
        $\quad M$, a list of all generating invariants\vspace{0.3cm}\\
    }
    \BlankLine\DontPrintSemicolon
    $M = \{\,\}$\tcc*{an empty list}
      \BlankLine
      
    $t$ = length $L$ \tcc*{number of seeds}
    
      \BlankLine
      
    \ForEach{$c \in \{(c_1, \ldots , c_t) \mid c_i \in \{1, \ldots, p-1 \}\}$}{
    
      \BlankLine
        $m' =  \sum c_i s_i \mod p$ \tcc*{combinations of seeds}
        \BlankLine
        \If{$(m')_j \leq (g)_j \, , \forall j , \, \forall  g \in L \cup M $} {
         \BlankLine
         \tcc*{If not divisible by gens so far}
          \If{$\sum_j (m')_j \leq k(p-1)+1$ } {
                \BlankLine
         \tcc*{or above Olson's bound}
        
         \BlankLine
                $M$.\code{append($m'$)}
                 \BlankLine
          \tcc*{add it as a generator}
            }
        }
       }
    \BlankLine
     $M$.\code{append($L$)}
         \BlankLine
    \ForEach{\tn{var }$x \in \code{gens(}R\code)$}{
        $M$.\code{append($x^p$)}
    }
    \BlankLine
    \tb{return }\code{reduce($M$)}
\end{algorithm}

\end{document}